\documentclass{article}

\usepackage[utf8]{inputenc}
\usepackage[T1]{fontenc}
\usepackage{amsmath}
\usepackage{amssymb}
\usepackage{amsfonts}
\usepackage{amsthm}
\usepackage{bm}
\usepackage{graphicx}
\usepackage{faktor}
\usepackage{yfonts}
\usepackage{bbding}
\usepackage{fancyvrb}
\usepackage{xargs}
\usepackage{hyperref}

\title{An independence of the MIN from the PHP principle}
\author{Mykyta Narusevych\footnote{Supported by the Charles University project PRIMUS/21/SCI/014, Charles University Research Centre program No. UNCE/24/SCI/022 and GA UK project No. 246223}}
\date{Faculty of Mathematics and Physics \\
Charles University\footnote{Sokolovsk\'a 83, Prague, 186 75, The Czech Republic}}

\begin{document}

\newcommand{\MIN}{\textsf{MIN}(\triangleleft)}
\newcommand{\T}{\textsf{T}^1_2}
\newcommand{\Talpha}{\textsf{T}^1_2(\alpha)}
\newcommand{\Ttwo}{\textsf{T}^2_2}
\newcommand{\Tmin}{\textsf{T}^1_2(\triangleleft)}
\newcommand{\Smin}{\textsf{S}^1_2(\triangleleft)}
\newcommand{\Dmin}{\Delta^b_1(\triangleleft)}
\newcommand{\M}{\mathbb{M}}
\newcommand{\I}{\mathbb{I}}
\renewcommand{\S}{\mathcal{S}}
\renewcommand{\O}{\mathbb{O}}
\newcommand{\TOUR}{\textsf{TOUR}}
\newcommand{\dWPHP}{\textsf{dWPHP}}
\newcommand{\deff}[1]{\textbf{#1}}
\newcommand{\el}{\operatorname{el}}
\newcommand{\PHPD}[2]{\textsf{PHP}^{#1}_{#2}(\Dmin)}

\newcommand{\revisionnote}[1]{%
\par\medskip\noindent\fbox{\begin{minipage}{0.94\linewidth}
\textbf{Revision marker.} #1
\end{minipage}}\par\medskip}

\newcommand{\humanjudgement}[1]{%
\par\medskip\noindent\fbox{\begin{minipage}{0.94\linewidth}
\textbf{Human judgement required.} #1
\end{minipage}}\par\medskip}

\theoremstyle{plain}
\newtheorem{theorem}{Theorem}[section]
\newtheorem{lemma}[theorem]{Lemma}
\newtheorem{corollary}[theorem]{Corollary}
\newtheorem{proposition}[theorem]{Proposition}
\newtheorem*{fact}{Fact}

\theoremstyle{definition}
\newtheorem{definition}[theorem]{Definition}
\newtheorem{problem}[theorem]{Problem}

\theoremstyle{remark}
\newtheorem*{remark}{Remark}
\newtheorem*{conjecture}{Conjecture}

\maketitle

\begin{abstract}
The minimization principle $\MIN$, studied in bounded arithmetic in \cite{chiarikrajicek98}, says that a strict linear ordering $\triangleleft$ on any finite interval $[0,\dots,n)$ has the least element. We prove that the relativized bounded arithmetic theory $\Tmin$ augmented by pigeonhole-principle instances for all $\Dmin$ formulas does not prove $\textsf{MIN}(\triangleleft)$.
\end{abstract}

\section{Introduction}

One of the central themes of propositional proof complexity is to understand the relative strength of proof systems and to prove strong lower bounds on proof length. Bounded arithmetic provides a uniform framework for studying such questions and often allows one to reduce proof-complexity lower bounds to logical independence results over corresponding arithmetic theories; see \cite[Section 8]{krajicek19}.

Motivating open problems for the present work include the $d$ versus $d+1$ problem \cite[Section 12.2]{krajicek95} and the $R(\log)$ problem \cite[Problem 13.7.1]{krajicek19}. From the bounded arithmetic perspective, the former asks for separations between adjacent theories $\textsf{T}^{i+1}_2(\alpha)$ and $\textsf{T}^i_2(\alpha)$ by formulas of low complexity. The latter can be viewed as a particularly basic special case: it asks for a $\Sigma^b_1(\alpha)$ formula provable in the full hierarchy $\textsf{T}_2(\alpha)$ but not in the low fragment $\textsf{T}^2_2(\alpha)$. See Section~\ref{preliminaries} for the relevant definitions. At the same time, the lowest fragment $\textsf{T}^1_2(\alpha)$ is relatively well understood, especially due to the theorem of Riis \cite{riis93}, \cite[Theorem 11.3.2]{krajicek95}. We believe that one possible route towards the above problems is to develop methods for proving independence from non-trivial extensions of $\textsf{T}^1_2(\alpha)$.

The minimization principle was introduced in bounded arithmetic in \cite{chiarikrajicek98} in the study of definable search problems and witnessing theorems. It is a $\Sigma^b_2$ formula $\MIN$ expressing that if $\triangleleft$ is a strict linear ordering of the set $\{0,\dots,n-1\}$, then it has a least element. It defines a total $\Sigma^p_2$ search problem: either certify that $\triangleleft$ is not a strict linear order, or find its least element.

The pigeonhole principle $\textsf{PHP}$ has an especially rich history in bounded arithmetic and proof complexity; see, for example, \cite[Section 12.5]{krajicek95}, \cite[Section 15]{krajicek19}, and \cite{ajtai88}. Given a binary relation $R$, it states that $R$ cannot define the graph of an injective mapping from $\{0,\dots,n\}$ into $\{0,\dots,n-1\}$.

In this paper, we study the extension of $\textsf{T}^1_2(\triangleleft)$ by instances of the pigeonhole principle for polynomial-time, $\triangleleft$-oracle computable graphs of functions, equivalently for graphs defined by $\Dmin$ formulas. We show that this theory does not prove the minimization principle for $\triangleleft$:
\[
    \Tmin + \textsf{PHP}(\Dmin) \nvdash \forall n\,\MIN_n.
\]
Since $\MIN$ is provable in $\textsf{T}^2_2(\triangleleft)$ \cite{chiarikrajicek98}, this gives a principle provable in $\textsf{T}^2_2(\triangleleft)$ but not in $\Tmin+\textsf{PHP}(\Dmin)$. Conversely, the ordinary pigeonhole principle is not provable even in strong bounded arithmetic theory $\textsf{T}_2(\alpha)$ \cite{krajicekpudlakwoods95,pitassibeameimpagliazzo95}. Thus, in the relevant relativized setting, the extension $\Tmin+\textsf{PHP}(\Dmin)$ and the theory $\textsf{T}^2_2(\triangleleft)$ are incomparable in strength.

It is worth noting that if one replaces the pigeonhole principle by its weak version, the weak pigeonhole principle \cite[Section 11.2]{krajicek95}, the corresponding separation follows directly from \cite{muller20}. There it is shown that, over $\Talpha$, no \textit{strong} combinatorial principle, such as the minimization principle, follows from \textit{weak} combinatorial principles, such as the weak pigeonhole principle.

Our primary technique is a model-theoretic construction reminiscent of \textit{forcing}. Conditions are small pieces of the oracle relation, and the desired expansion is obtained by meeting countably many requirements. We present the construction as a game with three kinds of players. One player makes the oracle order falsify $\MIN$, one player preserves induction, and one player ensures all $\Dmin$ pigeonhole-principle instances. The main technical point is that the pigeonhole-principle requirements can always be met. This is reduced to a combinatorial statement about arrays of partial oracle conditions.

We present two additional separations to further demonstrate how the same combinatorial method can be applied. Concretely, we show how to obtain models of $\Talpha$ plus the $\Delta^b_1(\alpha)$ pigeonhole principle that fail the dual weak pigeonhole principle and the tournament principle; see Section~\ref{applications} for definitions.

The paper is organized as follows.

In Section~\ref{preliminaries}, we introduce the notions relevant to the paper: the relativized bounded arithmetic theory $\Talpha$, the minimization principle $\textsf{MIN}$, and the pigeonhole principle $\textsf{PHP}$.

In Section~\ref{model}, we outline a game-theoretic construction of a model of $\Tmin$ satisfying the $\Dmin$ pigeonhole principle. We also briefly discuss how the construction can be viewed through the lens of forcing.

In Section~\ref{analysis}, we prove the main combinatorial theorems needed to finish the model construction. We also explain how the construction can be modified to produce a model with $\triangleleft$ interpreted as a dense linear ordering.

In Section~\ref{applications}, we sketch additional separations using the machinery developed in the previous sections. In particular, we discuss the tournament principle and the dual weak pigeonhole principle. Finally, we formulate an open problem related to the methods used in the paper.

\section{Preliminaries}
\label{preliminaries}

\subsection{Bounded Arithmetic}

In this subsection, we provide a brief exposition of bounded arithmetic. For further details, see \cite{buss85} and \cite{krajicek95}.

We study the relativized bounded arithmetic theory $\Tmin$, where $\triangleleft$ is a binary relation symbol. In addition to $\triangleleft$, the language contains
\begin{itemize}
    \item constant symbols $0$ and $1$,
    \item unary operations $\lfloor\cdot/2\rfloor$ and $|\cdot|$,
    \item binary operations $+$, $\cdot$, and $\#$,
    \item the binary relation $\leq$.
\end{itemize}
The intended meanings of $0,1,\lfloor\cdot/2\rfloor,+,\cdot,\leq$ are standard. The symbol $|\cdot|$ denotes bit-length, namely $|0|=0$ and, for $x>0$, $|x|=\lfloor\log_2 x\rfloor+1$. The symbol $\#$ denotes the binary \textit{smash function}, standardly defined as $2^{|x|\cdot |y|}$.

A particular theory in this language is $\textsf{BASIC}$; see \cite[Definition 5.2.1]{krajicek95}. This finite theory describes the basic properties of the involved symbols, such as commutativity of $+, \cdot$, and recursive properties of $\#$.

A \deff{bounded formula} is a formula logically equivalent to one containing only bounded quantifiers, that is, quantifiers of the form $\forall x\leq t$ or $\exists x\leq t$, where the term $t$ does not contain $x$.

The class $\Delta^b_0(\triangleleft)$ denotes sharply bounded formulas in the language with the additional relation symbol $\triangleleft$. Here sharply bounded means that all quantified variables are bounded by terms of the form $|t|$.

The class $\Sigma^b_1(\triangleleft)$ is defined as the smallest class of bounded formulas extending $\Delta^b_0(\triangleleft)$ and satisfying the following closure properties:
\begin{itemize}
    \item if $\varphi,\psi\in\Sigma^b_1(\triangleleft)$, then $\varphi\land\psi$ and $\varphi\lor\psi$ are in $\Sigma^b_1(\triangleleft)$;
    \item if $\varphi\in\Sigma^b_1(\triangleleft)$ and $t$ is a term, then $\forall x\leq |t|\,\varphi$ and $\exists x\leq |t|\,\varphi$ are in $\Sigma^b_1(\triangleleft)$;
    \item if $\varphi\in\Sigma^b_1(\triangleleft)$ and $t$ is a term, then $\exists x\leq t\,\varphi$ is in $\Sigma^b_1(\triangleleft)$.
\end{itemize}
The class $\Pi^b_1(\triangleleft)$ is the set of formulas $\varphi$ such that $\neg\varphi\in\Sigma^b_1(\triangleleft)$.

For a formula $\varphi(x,\overline y)$, let $\textsf{PIND}(\varphi)$ denote the sentence
\begin{align*}
    \forall\overline y\,\bigl[(\varphi(0,\overline y)\land
    \forall x(\varphi(\lfloor x/2\rfloor,\overline y)\to\varphi(x,\overline y)))
    \to \forall x\,\varphi(x,\overline y)\bigr].
\end{align*}
As an axiom scheme, this is called \deff{polynomial induction}. The usual induction scheme, in which the inductive step goes from $x$ to $x+1$, is denoted by $\textsf{IND}$.

The theory $\textsf{S}^1_2(\triangleleft)$ is $\textsf{BASIC}$ augmented by polynomial induction for $\Sigma^b_1(\triangleleft)$ formulas. The theory $\Tmin$ extends $\textsf{S}^1_2(\triangleleft)$ by the usual induction scheme for $\Sigma^b_1(\triangleleft)$ formulas.

Finally, $\Dmin$ denotes the class of formulas $\varphi\in\Sigma^b_1(\triangleleft)$ for which there is a formula $\psi\in\Sigma^b_1(\triangleleft)$ such that
\[
    \textsf{S}^1_2(\triangleleft)\vdash \varphi\leftrightarrow \neg\psi.
\]
Note that the above formula class is often referred to as $\Dmin$ \textit{in} $\Smin$ to emphasize the base theory over which the equivalence is established.

\begin{theorem}[{Buss \cite{buss85}; see also \cite[Section 7]{krajicek95}}]
\label{buss}
    A relation $R(\overline x)$ on numbers is computable in polynomial time with $\triangleleft$-oracle access if and only if it is definable by a $\Delta^b_1(\triangleleft)$ formula $\varphi(\overline x)$.
\end{theorem}

\subsection{$\textsf{MIN}$ and $\textsf{PHP}$ principles}
\label{subsection min}

We write $[0,\dots,n)$ for the set $\{0,\dots,n-1\}$.

The following sentence $\MIN_n$ expresses that, assuming $\triangleleft$ defines a strict linear ordering on $[0,\dots,n)$, there exists a $\triangleleft$-least element of $[0,\dots,n)$:
\begin{gather*}
    \exists x<n\,(x\triangleleft x)
    \\
    \lor
    \\
    \exists x,y<n\,(x\neq y\land x\triangleleft y\land y\triangleleft x)
    \\
    \lor
    \\
    \exists x,y<n\,(x\neq y\land \neg(x\triangleleft y)\land \neg(y\triangleleft x))
    \\
    \lor
    \\
    \exists x,y,z<n\,(x\triangleleft y\land y\triangleleft z\land \neg(x\triangleleft z))
    \\
    \lor
    \\
    \exists x<n\,\forall y<n\,(x=y\lor x\triangleleft y).
\end{gather*}
This principle was introduced in bounded arithmetic in \cite{chiarikrajicek98}, where it was shown that $\MIN_n$ is provable in $\textsf{T}^2_2(\triangleleft)$ but is not provable in $\Tmin$.

Let $\varphi(x,y,\overline z)$ be a formula. We denote by $\textsf{PHP}(\varphi)$ the pigeonhole principle for $\varphi$. It is the sentence
\[
    \forall\overline z\,\forall u,v\,(u>v\to \psi(u,v,\overline z)),
\]
where $\psi(u,v,\overline z)$ is
\begin{gather*}
    \exists x<u\,\forall y<v\,\neg\varphi(x,y,\overline z)
    \\
    \lor
    \\
    \exists x,x'<u\,\exists y<v\,(x\neq x'\land \varphi(x,y,\overline z)\land \varphi(x',y,\overline z))
    \\
    \lor
    \\
    \exists x<u\,\exists y,y'<v\,(y\neq y'\land \varphi(x,y,\overline z)\land \varphi(x,y',\overline z)).
\end{gather*}
In words, this says that $\varphi$ cannot be the graph of an injective function from $u$ pigeons into $v$ holes when $u>v$.

The pigeonhole principle is famously known to be unprovable in the full bounded arithmetic theory $\textsf{T}_2(R)$ \cite{ajtai88,krajicekpudlakwoods95,pitassibeameimpagliazzo95}.

We denote the set of instances $\textsf{PHP}(\varphi)$ for all $\Dmin$ formulas $\varphi$ by $\textsf{PHP}(\Dmin)$. The main result of the paper is the following theorem.

\begin{theorem}
\label{main}
\[
    \Tmin+\textsf{PHP}(\Dmin)\nvdash \forall n\,\MIN_n.
\]
\end{theorem}

\section{Model construction}
\label{model}

\subsection{The game $G_\O$}

In this subsection, we construct a model expansion, later proved to satisfy $\Tmin+\textsf{PHP}(\Dmin)$ and to violate $\MIN_n$ for some $n$.

We start with a countable nonstandard model $\M$ of true arithmetic in the language $\{0,1,+,\cdot,|\cdot|,\lfloor\cdot/2\rfloor,\#,\leq\}$. Pick a nonstandard number $n$ in $\M$ and let $\I$ be the substructure of $\M$ with domain
\[
    \{m: m\leq 2^{|n|^c}\text{ for some standard }c\}.
\]

Our goal is to interpret $\triangleleft$ in $\I$ suitably. This is achieved through a step-by-step construction, which we view as a game played between three parties, denoted by $P_{\textsf{MIN}}$, $P_{\textsf{IND}}$, and $P_{\textsf{PHP}}$. The game runs for countably many steps. Its moves are sequences definable in the ground model $\M$, while the play itself (i.e., the limit sequence of all moves) is external and need not be definable in $\M$.

The game starts with $\triangleleft$ being the empty relation. The players subsequently extend $\triangleleft$ by adding at most $|n|^c$ many elements at each move, for some standard $c$, with $c$ possibly varying between the moves. The invariant is that $\triangleleft$ is always a strict linear ordering of a subset of $[0,\dots,n)$.

\begin{definition}
    Let $\O$ be the set of sequences of distinct numbers from $[0,\dots,n)$ of length at most $|n|^c$, for some standard constant $c$.

    Each $o\in\O$ induces a strict partial ordering of $[0,\dots,n)$, denoted by $\triangleleft_o$, in the obvious way: the domain consists of the numbers appearing in $o$, and their ordering is the order in which they are listed in $o$. We write $\el(o)$ for the domain of $o$.

    We say that $o'\in\O$ \deff{extends} $o\in\O$, denoted $o'\preceq o$ or $o\succeq o'$, if $\triangleleft_{o'}\supseteq \triangleleft_o$. Thus $\preceq$ is the forcing order: stronger conditions are smaller.

    We say that $o'\in\O$ \deff{is compatible with} $o\in\O$, denoted $o'\|o$, if $\triangleleft_{o'}$ is compatible with $\triangleleft_o$ as a strict ordering. Equivalently, there is $o''\in\O$ extending both $o$ and $o'$. Incompatibility is denoted by $o\perp o'$.
\end{definition}

\begin{definition}
    The \deff{game} $G_\O$ is played by three players $P_{\textsf{MIN}}$, $P_{\textsf{IND}}$, and $P_{\textsf{PHP}}$ for countably many steps.

    The player $P_{\textsf{MIN}}$ plays at stages $k\equiv 0\pmod 3$, the player $P_{\textsf{IND}}$ plays at stages $k\equiv 1\pmod 3$, and the player $P_{\textsf{PHP}}$ plays at stages $k\equiv 2\pmod 3$.

    Let $o_0$ be the empty sequence. At stage $k$, the respective player is presented with $o_k\in\O$, chooses $o_{k+1}\preceq o_k$ from $\O$, and passes $o_{k+1}$ to the next player.

    A \deff{strategy} $\S$ is a function that, on input $o\in\O$, outputs $o'\in\O$ such that $o'\preceq o$.

    Fixing strategies $\S_{\textsf{MIN}}$, $\S_{\textsf{IND}}$, and $\S_{\textsf{PHP}}$ induces a chain
    \[
        o_0\succeq o_1\succeq o_2\succeq\cdots.
    \]
    This chain induces the strict ordering
    \[
        \triangleleft_{\S_{\textsf{MIN}},\S_{\textsf{IND}},\S_{\textsf{PHP}}}
        :=\bigcup_i \triangleleft_{o_i}.
    \]
    We drop the subscripts when the strategies are clear from context.

    The expansion of $\I$ by interpreting $\triangleleft$ as this union is denoted by $\I_\triangleleft$.

    The goal of $P_{\textsf{MIN}}$ is for the resulting interpretation of $\triangleleft$ to be a strict linear order of $[0,\dots,n)$ with no least element, no matter how the other players play.

    The goal of $P_{\textsf{IND}}$ is for the resulting structure $\I_\triangleleft$ to satisfy induction for $\Sigma^b_1(\triangleleft)$ formulas, no matter how the other players play.

    The goal of $P_{\textsf{PHP}}$ is for the resulting structure $\I_\triangleleft$ to satisfy $\textsf{PHP}(\Dmin)$, no matter how the other players play.

    A strategy of one of the players is called \deff{winning} if the corresponding goal is achieved regardless of how the other players play.
\end{definition}

\subsection{Winning strategies}

\begin{proposition}
\label{Pmin winning strategy}
    There exists a strategy $\S_{\textsf{MIN}}$ winning for $P_{\textsf{MIN}}$.
\end{proposition}

\begin{proof}
    No matter how the players proceed, the resulting interpretation of $\triangleleft$ is a strict, though not necessarily total, ordering of a subset of $[0,\dots,n)$.

    To ensure that $\triangleleft$ is total, $P_{\textsf{MIN}}$ maintains an external countable queue of all elements of $[0,\dots,n)$, possible because $\M$ is countable externally.

    Given a sequence $o=(a_1,\dots,a_m)\in\O$, the player removes all $a_i$ from the queue, picks the first remaining element $a$, and extends $o$ to
    \[
        o'=(a,a_1,\dots,a_m).
    \]
    Every element of $[0,\dots,n)$ is eventually inserted, and every inserted element eventually has a smaller element placed before it. Hence the final ordering is total on $[0,\dots,n)$ and has no least element.
\end{proof}

The theorem below is a direct adaptation of a classical result of Paris and Wilkie \cite[Theorem 21]{pariswilkie85} for the theory $I\exists_1(\alpha)$, later extended by Riis \cite{riis93} to $\Talpha$. Here we provide only a sketch of the proof, more details and presentations of the mentioned theorem can be found, e.g., in \cite[Section 12.7]{krajicek95} and \cite[Theorem 2.1]{atseriasmuller15}.

\begin{theorem}
\label{Pind winning strategy}
    There exists a strategy $\S_{\textsf{IND}}$ winning for $P_{\textsf{IND}}$.
\end{theorem}

\begin{proof}[Proof (sketch).]
We use the standard fact that $\Sigma^b_1(\triangleleft)$ induction is
equivalent to the \emph{least number principle} for $\Sigma^b_1(\triangleleft)$ formulas. Thus it is enough to ensure, for every $\Sigma^b_1(\triangleleft)$ formula $\varphi(x,\bar a)$ with parameters from $\mathbb I$ and every bound $t\in\mathbb I$,
the sentence
\[
    \exists x<t\,\varphi(x,\bar a)
    \rightarrow
    \exists x<t\,
    \bigl(\varphi(x,\bar a)\land
    \forall y<x\,\neg\varphi(y,\bar a)\bigr).
\]

Say that a sequence $o\in\mathbb O$ \emph{settles} a sentence $\theta$ positively if every legal continuation of the play starting from $o$ produces a final expansion satisfying $\theta$.

We use the following observation. Let $\theta$ be a
$\Sigma^b_1(\triangleleft)$ sentence with parameters from $\mathbb I$. If some legal continuation of the play from $o$ produces a final expansion satisfying $\theta$, then there is an extension $o'\preceq o$ which settles $\theta$ positively.

Now fix the current condition $o$, a formula $\varphi(x,\bar a)$, and a bound $t\in\mathbb I$. Define
\[
    X=\{x<t:\text{ some }o'\preceq o
    \text{ settles }\varphi(x,\bar a)\text{ positively}\}.
\]
The technical content of the Paris--Wilkie--Riis theorem is that sets of this kind are
definable in the ground model $\M$.

If $X$ is empty, then no legal continuation from $o$ can make
$\varphi(x,\bar a)$ true for any $x<t$. Hence $o$ already guarantees the negation of
the antecedent
\[
    \exists x<t\,\varphi(x,\bar a),
\]
and therefore guarantees the least-number-principle implication.

If $X$ is non-empty, let $x_0$ be the least element of $X$. Choose
$o'\preceq o$ which settles $\varphi(x_0,\bar a)$ positively. By the minimality of $x_0$, no $y<x_0$ belongs to $X$. By the observation above, this means that no legal continuation from $o$, and hence no legal continuation from $o'$, can make
$\varphi(y,\bar a)$ true for any $y<x_0$. Therefore $o'$ guarantees
\[
    \varphi(x_0,\bar a)\land
    \forall y<x_0\,\neg\varphi(y,\bar a).
\]
Thus $o'$ guarantees the conclusion of the least-number principle.
\end{proof}

To finish the proof of Theorem~\ref{main}, it is enough to prove the following theorem.

\begin{theorem}
\label{main strategy}
    There exists a strategy $\S_{\textsf{PHP}}$ winning for $P_{\textsf{PHP}}$.
\end{theorem}

\subsection{MIN-trees}

\begin{definition}
\label{min tree}
    Let $o\in\O$. A subset $T$ of $\O$ definable in $\M$ is called an \deff{$o$-MIN-tree}, or simply an \deff{$o$-tree}, if it satisfies the following recursive definition:
    \begin{itemize}
        \item either $T=\{o\}$;
        \item or, writing $o=(a_1,\dots,a_m)$, there is a single $a\in[0,\dots,n)\setminus\{a_1,\dots,a_m\}$ such that
        \[
            T=\bigcup_{o'\in A}T_{o'},
        \]
        where $A$ is the set of all sequences extending $o$ by adding $a$ to its domain, and each $T_{o'}$ is an $o'$-tree.
    \end{itemize}

    The \deff{depth} of $T$ is the largest difference between the length of a sequence in $T$ and the length of $o$. An $o$-tree is \deff{uniform} if all its members have the same length. The \deff{size} of $T$ is the number of elements it contains.
\end{definition}

\begin{remark}
    The above definition treats trees as definable subsets of $\O$ that can be suitably factorized. A more standard convention, such as the definition of PHP-trees in \cite[Section 15.1]{krajicek19}, keeps the entire tree structure. The present convention is sufficient here because only the maximal conditions represented by the tree are used.
\end{remark}

\begin{proposition}
\label{min tree properties}
    Let $T$ be an $o$-tree.
    \begin{enumerate}
        \item Any two distinct elements of $T$ are incompatible.
        \item Any member of $\O$ compatible with $o$ is compatible with some element of $T$.
        \item If the depth of $T$ is at most $d$, then
        \[
            |T|\leq \frac{(|o|+d)!}{|o|!},
        \]
        with equality if and only if $T$ is a uniform $o$-tree of depth exactly $d$.
    \end{enumerate}
\end{proposition}

By Theorem~\ref{buss}, every $\Delta^b_1(\triangleleft)$ relation is computable by a polynomial-time oracle algorithm with oracle $\triangleleft$. Fixing numerical parameters from $\I$, the possible oracle computations form a decision tree of depth polynomial in the lengths of those parameters. Since throughout the construction we maintain that $\triangleleft$ is a strict partial order, an oracle query of the form $u\triangleleft v$ can equivalently be implemented by inserting not-yet-seen elements among the currently ordered ones with their relative ordering given by the query answer. Thus the usual decision tree can be represented as a MIN-tree.

\begin{proposition}
\label{buss restatement}
    Let $\varphi(x,y)$ be a $\Dmin$ formula with parameters from $\I$. Let $p,h\in\I$ and $o\in\O$ be fixed.

    There exist definable families
    \[
        \{T^+_{a,b}\}_{a<p,b<h}
        \quad\text{and}\quad
        \{T^-_{a,b}\}_{a<p,b<h}
    \]
    such that, for each $a<p$ and $b<h$, the sets $T^+_{a,b}$ and $T^-_{a,b}$ are disjoint and their union is an $o$-tree of depth at most $|n|^c$, for a fixed standard $c$ depending only on $\varphi$.

    Furthermore, $\Smin$ proves that, for all $a<p$ and $b<h$,
    \begin{align*}
        (\neg\MIN_n\land\text{``$\triangleleft$ extends $o$''})
        \to
        \bigl(\varphi(a,b)\leftrightarrow
        \text{``$\triangleleft$ extends some $o'\in T^+_{a,b}$''}\bigr),
    \end{align*}
    and
    \begin{align*}
        (\neg\MIN_n\land\text{``$\triangleleft$ extends $o$''})
        \to
        \bigl(\neg\varphi(a,b)\leftrightarrow
        \text{``$\triangleleft$ extends some $o'\in T^-_{a,b}$''}\bigr).
    \end{align*}
\end{proposition}

\begin{theorem}
\label{single step}
    Let $\varphi(x,y)$ be a $\Dmin$ formula with parameters from $\I$. Let $p,h\in\I$ and $o\in\O$ be fixed.

    Let $\{T^+_{a,b}\}_{a<p,b<h}$ be as in Proposition~\ref{buss restatement}. Suppose that at least one of the following properties holds:
    \begin{itemize}
        \item there exist $a\neq b<p$ and $c<h$ such that, for some $o'\in T^+_{a,c}$ and $o''\in T^+_{b,c}$, we have $o'\|o''$;
        \item there exist $a<p$ and $b\neq c<h$ such that, for some $o'\in T^+_{a,b}$ and $o''\in T^+_{a,c}$, we have $o'\|o''$;
        \item there exist $o'\preceq o$ and $a<p$ such that, for all $b<h$ and all $o''\in T^+_{a,b}$, we have $o'\perp o''$.
    \end{itemize}
    Then there exists $o^*\in\O$ extending $o$ such that, assuming $P_{\textsf{MIN}}$ uses $\S_{\textsf{MIN}}$ as in Proposition~\ref{Pmin winning strategy} and $P_{\textsf{IND}}$ uses $\S_{\textsf{IND}}$ as in Theorem~\ref{Pind winning strategy}, the formula $\varphi(x,y)$ does not define the graph of an injective function from $[0,\dots,p)$ into $[0,\dots,h)$ in any expansion $(\I,\triangleleft)$ constructed in a play where $P_{\textsf{PHP}}$ plays $o^*$.
\end{theorem}

\begin{proof}
    The assumptions on $P_{\textsf{MIN}}$ and $P_{\textsf{IND}}$ guarantee that the resulting expansion of $\I$ satisfies $\Tmin$ and that $\triangleleft$ is a strict linear ordering of $[0,\dots,n)$.

    If one of the first two bullets holds, let $o^*$ be any common extension of the compatible conditions $o'$ and $o''$. Then $o^*$ ensures violation of either injectivity or functionality of the corresponding pigeonhole-principle instance.

    Suppose the third bullet holds, and let $o^*=o'$. We claim that this settles
    \[
        \exists x<p\,\forall y<h\,\neg\varphi(x,y).
    \]
    Let $a<p$ be such that, for all $b<h$ and all $o''\in T^+_{a,b}$, we have $o^*\perp o''$.

    For contradiction, suppose there is $b<h$ such that the resulting structure satisfies
    \[
        \I_\triangleleft\models\varphi(a,b).
    \]
    Since $\I_\triangleleft$ models $\Smin$, Proposition~\ref{buss restatement} gives some $o''\in T^+_{a,b}$ such that $\triangleleft$ extends $o''$. Since $\triangleleft$ also extends $o^*$, the conditions $o^*$ and $o''$ are compatible, contradicting the choice of $o^*$.
\end{proof}

Assuming we can establish that, whenever $p>h$, the family $\{T^+_{a,b}\}_{a<p,b<h}$ satisfies one of the hypotheses of Theorem~\ref{single step}, we can prove Theorem~\ref{main}. Namely, the winning strategy $\S_{\textsf{PHP}}$ for $P_{\textsf{PHP}}$ enumerates all $\Dmin$ formulas with parameters from $\I$ and all pairs $p>h$ from $\I$. Being presented with $o\in\O$, they choose the next formula and the next pair $p>h$, apply Theorem~\ref{single step}, and extend $o$ to the corresponding $o^*$.

\subsection{A note on forcing}

It is possible to view the construction presented in this paper through the lens of forcing. The forcing frame is $\O$, and a generic object is an interpretation of $\triangleleft$ obtained by meeting the relevant dense sets. One then shows that any sufficiently generic interpretation of $\triangleleft$ satisfies the required conditions.

A more detailed presentation can be found in \cite{narusevych24}, with an application to a related problem. More information on this kind of forcing in bounded arithmetic can be found in \cite{krajicek95,atseriasmuller15,muller20}.

\section{Combinatorial analysis}
\label{analysis}

\subsection{PHP-arrays}

In this section, we show that any definable family as in Proposition~\ref{buss restatement} satisfies the hypothesis of Theorem~\ref{single step} whenever $p>h$.

\begin{definition}
\label{min array}
    Let $o\in\O$ and $p,h\in\I$. Let
    \[
        A=\{A_{a,b}\}_{a<p,b<h}
    \]
    be an indexed family, definable in $\M$, such that each $A_{a,b}$ is a subset of $\O$.

    We say that $A$ is an \deff{$(o,p,h)$-PHP-array}, or simply an \deff{$(o,p,h)$-array}, if the following conditions are satisfied:
    \begin{itemize}
        \item for all $a<p$ and $b<h$, every $q\in A_{a,b}$ extends $o$;
        \item for all $a<p$ and $b<h$, any two distinct $q,r\in A_{a,b}$ are incompatible;
        \item for all $a\neq a'<p$ and $b<h$, every $q\in A_{a,b}$ is incompatible with every $r\in A_{a',b}$;
        \item for all $a<p$ and $b\neq b'<h$, every $q\in A_{a,b}$ is incompatible with every $r\in A_{a,b'}$;
        \item for every $q\preceq o$ and every $a<p$, there are $b<h$ and $r\in A_{a,b}$ such that $q\|r$.
    \end{itemize}
\end{definition}

\begin{remark}
    Since $A$ is definable in $\M$, there is a fixed standard constant $c$ such that all sequences from all $A_{a,b}$ have length at most $|n|^c$. This follows by overspill; see \cite[Section 6.1]{kaye91}.
\end{remark}

\begin{remark}
    Although $A$ as an indexed family is assumed to be definable, the statement ``$A$ is an $(o,p,h)$-array'' is external. The reason is the last bullet of Definition~\ref{min array}, which quantifies over all elements of $\O$.
\end{remark}

\begin{lemma}
\label{min array exist}
    Let $\varphi(x,y)$ be a $\Dmin$ formula with parameters from $\I$. Let $p,h\in\I$ and $o\in\O$ be fixed.

    Let $A$ denote the indexed family $\{T^+_{a,b}\}_{a<p,b<h}$ from Proposition~\ref{buss restatement}. Assume $A$ does not satisfy any of the properties listed in Theorem~\ref{single step}.

    Then $A$ is an $(o,p,h)$-array.
\end{lemma}

\begin{proof}
    The first two properties follow from the fact that each $T^+_{a,b}$ is a subset of an $o$-tree. The third and fourth properties are precisely the negations of the first two alternatives in Theorem~\ref{single step}. The fifth property is the negation of the third alternative in Theorem~\ref{single step}.
\end{proof}

\begin{definition}
\label{min array size}
    Let $A$ be an $(o,p,h)$-array. The \deff{size} of $A$ is
    \[
        |A|:=\sum_{a<p,b<h}|A_{a,b}|,
    \]
    where $|\cdot|$ denotes cardinality in $\M$.
\end{definition}

\begin{proposition}
\label{array row column size}
    For an $(o,p,h)$-array $A$,
    \[
        |A|=\sum_{a<p}\left|\bigcup_{b<h}A_{a,b}\right|
        =\sum_{b<h}\left|\bigcup_{a<p}A_{a,b}\right|.
    \]
\end{proposition}

\begin{proof}
    The equalities follow directly from the fact that any two elements from different cells of the same row or of the same column are distinct; indeed, they are incompatible.
\end{proof}

We write
\[
    A_a:=\bigcup_{b<h}A_{a,b}
    \quad\text{and}\quad
    A^b:=\bigcup_{a<p}A_{a,b}
\]
for the rows and columns of $A$.

\begin{definition}
    Given two $(o,p,h)$-arrays $A$ and $A'$, we say that $A'$ \deff{extends} $A$ if, for each $a<p$, $b<h$, and $q\in A'_{a,b}$, there is $r\in A_{a,b}$ such that $q\preceq r$.
\end{definition}

\begin{lemma}
\label{compatible envelope}
    Let $q,s\in\O$ be compatible conditions extending $o$. There is a $q$-tree $E(q,s)$, called \deff{envelope of} $s$ \deff{over} $q$, of depth at most $|s|-|o|$ such that every member $r\in E(q,s)$ satisfies
    \[
        \el(r)=\el(q)\cup\el(s).
    \]
    Moreover, if $s'\preceq o$ is incompatible with $s$ but compatible with $q$, then every $r\in E(q,s)$ compatible with $s'$ satisfies
    \[
        |\el(r)\cap\el(s')|\geq |\el(q)\cap\el(s')|+1.
    \]
\end{lemma}

\begin{proof}
    Construct $E(q,s)$ by successively choosing the elements of $\el(s)\setminus\el(q)$, in any fixed definable order, and at each step branching over all possible positions of the chosen element in the current sequence starting with $q$. This gives a $q$-tree. Its leaves are precisely the possible extensions of $q$ to the domain $\el(q)\cup\el(s)$.

    It remains to prove the last part of the claim. Suppose $s'\preceq o$, $s'\perp s$, and $s'\|q$. If
    \[
        (\el(s)\setminus\el(q))\cap\el(s')=\emptyset,
    \]
    then all elements common to $s$ and $s'$ already lie in $\el(q)$. Since both $s$ and $s'$ are compatible with $q$, the (partial) orderings induced by $s$ and by $s'$ agree on $\el(q) \supset \el(s) \cap \el(s')$. This contradicts $s\perp s'$.

    Therefore there is some element
    \[
        x\in (\el(s)\setminus\el(q))\cap\el(s').
    \]
    Every $r\in E(q,s)$ contains $\el(q)\cup\el(s)$, and hence contains this additional element $x$. Thus, whenever $r$ is compatible with $s'$,
    \[
        |\el(r)\cap\el(s')|\geq |\el(q)\cap\el(s')|+1.
    \]
\end{proof}

\begin{theorem}
\label{row uniformization}
    Let $A$ be an $(o,p,h)$-array. There exists an $(o,p,h)$-array $A'$ extending $A$ such that each row $A'_a$ is a uniform $o$-tree of depth $d=|n|^c$, for some fixed standard constant $c$.
\end{theorem}

\begin{proof}
    Let $d'=|n|^{c'}$ be such that every sequence appearing in $A$ has length at most $|o|+d'$. We fix $a<p$ and construct a sequence of $o$-trees
    \[
        T^a_0,T^a_1,\dots,T^a_{d'}
    \]
    satisfying the invariant
    \begin{equation}
    \label{intersection invariant}
        \forall q\in T^a_i\,\forall s\in A_a\,
        \bigl(q\|s\to |\el(q)\cap\el(s)|\geq \min(|o|+i,|s|)\bigr).
    \end{equation}
    For $i=0$, take $T^a_0=\{o\}$.

    Suppose $T^a_i$ has been constructed. For each leaf $q\in T^a_i$, proceed as follows.

    If $q$ extends some $s\in A_a$, then we leave $q$ unchanged at this stage. This preserves the invariant: if $s'\in A_a$ is compatible with $q$, then $s'=s$, because distinct elements of the same row are incompatible. Hence $q$ already contains all elements of $s'$.

    Otherwise, by the last property of an $(o,p,h)$-array, there exists some $s_q\in A_a$ compatible with $q$. Replace $q$ by the compatible envelope $E(q,s_q)$ from Lemma~\ref{compatible envelope}.

    Running this construction for every leaf $q\in T^a_i$ gives an $o$-tree $T^a_{i+1}$. We verify the invariant. Let $r\in T^a_{i+1}$ extend a leaf $q\in T^a_i$, and let $s'\in A_a$ be compatible with $r$. Then $s'$ is compatible with $q$.

    If $s'=s_q$, then $r$ contains $\el(s_q)$, hence
    \[
        |\el(r)\cap\el(s')|=|s'|.
    \]
    If $s'\neq s_q$, then $s'$ and $s_q$ are incompatible, since they are distinct elements of the same row $A_a$. By Lemma~\ref{compatible envelope},
    \[
        |\el(r)\cap\el(s')|\geq |\el(q)\cap\el(s')|+1.
    \]
    By the induction hypothesis,
    \[
        |\el(q)\cap\el(s')|\geq \min(|o|+i,|s'|).
    \]
    If this minimum were $|s'|$, then $q$ would contain all elements of $s'$, and since $q\|s'$, the condition $q$ would extend $s'$, contradicting the case we are considering. Hence the minimum is $|o|+i$, and so
    \[
        |\el(r)\cap\el(s')|\geq |o|+i+1
        =\min(|o|+i+1,|s'|).
    \]
    Thus the invariant holds for $T^a_{i+1}$.

    After $d'$ steps, for every $q\in T^a_{d'}$ and every $s\in A_a$ compatible with $q$, we have
    \[
        |\el(q)\cap\el(s)|=|s|.
    \]
    Hence $\el(s)\subseteq\el(q)$. Since $q$ and $s$ are compatible strict orderings, it follows that $q$ extends $s$.

    Extend $T^a_{d'}$ arbitrarily to a uniform $o$-tree $T^a$ of depth $d=|n|^c$, where $c$ is a standard constant large enough for all rows. The property that every leaf of the tree extends a unique element of $A_a$ is preserved. Uniqueness follows from the fact that distinct elements of $A_a$ are incompatible.

    For each $b<h$, define
    \[
        A'_{a,b}:=\{q\in T^a: \text{ for some }s\in A_{a,b},\ q\preceq s\}.
    \]
    Repeating the construction for every $a<p$ gives an indexed family $A'$.

    We verify that $A'$ is an $(o,p,h)$-array. Every element of every $A'_{a,b}$ extends $o$, because it belongs to the $o$-tree $T^a$. Any two distinct elements in the same cell are incompatible because they are distinct leaves of the same $o$-tree. The row and column incompatibility conditions follow because every element of $A'_{a,b}$ extends an element of $A_{a,b}$, and incompatibility is preserved under extension. Finally, if $q\preceq o$ and $a<p$, then, since $T^a$ is an $o$-tree, Proposition~\ref{min tree properties} gives a leaf $r\in T^a$ compatible with $q$. This leaf belongs to $A'_{a,b}$ for some $b<h$, and hence the last property of Definition~\ref{min array} holds.
\end{proof}

\begin{corollary}
\label{lower bound}
    An $(o,p,h)$-array $A$ is extendable to an $(o,p,h)$-array $A'$ of size exactly
    \[
        p\frac{(|o|+d)!}{|o|!},
    \]
    where $d=|n|^c$ for some standard $c$. Moreover, every element of every cell $A'_{a,b}$ has length exactly $|o|+d$.
\end{corollary}

\begin{proof}
    By Theorem~\ref{row uniformization}, extend $A$ to an array $A'$ such that each row $A'_a$ is a uniform $o$-tree of depth $d$. By Proposition~\ref{min tree properties}, each row has size exactly
    \[
        \frac{(|o|+d)!}{|o|!}.
    \]
    Summing over all $p$ rows gives the claimed size.
\end{proof}

Compared to \cite[Theorem 4.11]{narusevych24}, this gives an exact size rather than only a lower bound.

\begin{theorem}
\label{upper bound}
    Let $A$ be an $(o,p,h)$-array whose elements all have length at most $|o|+d$. Then
    \[
        |A|\leq h\frac{(|o|+d)!}{|o|!}.
    \]
\end{theorem}

\begin{proof}
    The proof is analogous to \cite[Theorem 4.12]{narusevych24}, which itself resembles the argument of \cite{katona72}.

    First extend every sequence appearing in $A$ arbitrarily to length exactly $|o|+d$, keeping it in the same cell. This does not change the cardinality of any cell and preserves incompatibility. Denote the resulting indexed family by $A'$.

    For each $b<h$, the column $(A')^b$ is a set of pairwise incompatible elements of $\O$, all of length exactly $|o|+d$. We claim that
    \[
        |(A')^b|\leq \frac{(|o|+d)!}{|o|!}.
    \]
    Choose $m$ in $\M$ large enough that every element appearing in every sequence from $(A')^b$ lies in $[0,\dots,m)$. There are exactly $m!/|o|!$ strict linear orderings of $[0,\dots,m)$ extending $o$. Each $q\in(A')^b$ extends to exactly $m!/(|o|+d)!$ such strict linear orderings. Since distinct elements of $(A')^b$ are incompatible, these sets of extensions are disjoint. Hence
    \[
        |(A')^b|\cdot \frac{m!}{(|o|+d)!}
        \leq \frac{m!}{|o|!},
    \]
    and the desired bound follows. Summing over all $h$ columns gives
    \[
        |A|=|A'|\leq h\frac{(|o|+d)!}{|o|!}.
    \]
\end{proof}

\begin{corollary}
\label{array inequality}
    If an $(o,p,h)$-array exists, then $p\leq h$.
\end{corollary}

\begin{proof}
    Extend $A$ to $A'$ as in Corollary~\ref{lower bound}. Then
    \[
        |A'|=p\frac{(|o|+d)!}{|o|!}.
    \]
    By Theorem~\ref{upper bound},
    \[
        |A'|\leq h\frac{(|o|+d)!}{|o|!}.
    \]
    Therefore $p\leq h$.
\end{proof}

\begin{proof}[Proof of Theorem~\ref{main strategy}]
    Let $\varphi(x,y)$ be a $\Dmin$ formula with parameters from $\I$, and let $p>h$ be elements of $\I$. Given $o\in\O$, consider the family $A=\{T^+_{a,b}\}_{a<p,b<h}$ from Proposition~\ref{buss restatement}. If $A$ satisfies one of the alternatives in Theorem~\ref{single step}, then $P_{\textsf{PHP}}$ can extend $o$ so that $\varphi$ fails to define an injective function from $p$ to $h$.

    If no such alternative holds, then by Lemma~\ref{min array exist}, $A$ is an $(o,p,h)$-array. Corollary~\ref{array inequality} then gives $p\leq h$, contradicting the choice of $p>h$. Therefore one of the alternatives in Theorem~\ref{single step} must always hold.

    Enumerating all $\Dmin$ formulas with parameters from $\I$ and all pairs $p>h$ from $\I$ gives a winning strategy for $P_{\textsf{PHP}}$.
\end{proof}

Combining Proposition~\ref{Pmin winning strategy}, Theorem~\ref{Pind winning strategy}, and Theorem~\ref{main strategy}, we obtain Theorem~\ref{main}.

\subsection{Densely ordered model}
\label{dense ordering}

The ordering constructed above can be made even more pathological. Namely, given a countable collection of sufficiently large subsets $X$ of $\I$, we can arrange that $\triangleleft$ orders all of them densely; that is, $(X,\triangleleft\vert_X)$ is a dense linear ordering. Here ``sufficiently large'' means that there is a definable injection in $\M$ mapping $[0,\dots,|n|^d)$ into $X$, for some nonstandard $d$.

We denote by $\textsf{DLO}(\triangleleft\vert_X)$ the following sentence:
\begin{gather*}
    \forall x\in X\,\neg(x\triangleleft x)
    \\
    \land
    \\
    \forall x,y\in X\,(x\neq y\to (x\triangleleft y\lor y\triangleleft x))
    \\
    \land
    \\
    \forall x,y,z\in X\,((x\triangleleft y\land y\triangleleft z)\to x\triangleleft z)
    \\
    \land
    \\
    \forall x\in X\,\exists y\in X\,(y\neq x\land y\triangleleft x)
    \\
    \land
    \\
    \forall x\in X\,\exists y\in X\,(y\neq x\land x\triangleleft y)
    \\
    \land
    \\
    \forall x,y\in X\,(x\triangleleft y\to \exists z\in X\,(x\triangleleft z\land z\triangleleft y)).
\end{gather*}

\begin{theorem}
\label{dense theorem}
    Assume $\{X_i\}_{i\in\mathbb N}$ is a countable collection of subsets of $\I$ such that, for each $X_i$, there exists a nonstandard $d$ and a definable injection mapping $[0,\dots,|n|^d)$ into $X_i$.

    Then there exists an expansion of $\I$ to a model of
    \[
        \Tmin+\textsf{PHP}(\Dmin)
    \]
    which moreover satisfies $\textsf{DLO}(\triangleleft\vert_{X_i})$ for all $i\in\mathbb N$.
\end{theorem}

\begin{proof}
    The model is constructed similarly to the proof of Theorem~\ref{main}. The set $\O$ is replaced by $\O'$, consisting of arbitrary sequences of elements of $\I$ of length at most $|n|^c$, for some standard constant $c$.

    The players $P_{\textsf{IND}}$ and $P_{\textsf{PHP}}$ are defined analogously, and the existence of their winning strategies is proved in the same way. For $P_{\textsf{PHP}}$, note that the lower and upper bounds on the sizes of arrays were proved without using the fact that the order was being constructed only on $[0,\dots,n)$.

    The strategy for $P_{\textsf{MIN}}$ is replaced by a countable sequence of players $P^{X_i}_{\textsf{DLO}}$, one for each $X_i$. The game is scheduled so that $P_{\textsf{IND}}$, $P_{\textsf{PHP}}$, and each $P^{X_i}_{\textsf{DLO}}$ play infinitely often.

    The player $P^X_{\textsf{DLO}}$ maintains an external queue of elements of $X$. Given a condition $o\in\O'$, let
    \[
        o_X=(a_{i_1},\dots,a_{i_m})
    \]
    be the subsequence of $o$ consisting of elements from $X$. The player removes these elements from its queue and chooses fresh elements
    \[
        b_0,b_1,\dots,b_m\in X.
    \]
    It then extends the induced ordering on $X$ so that
    \[
        b_0\triangleleft a_{i_1}\triangleleft b_1\triangleleft a_{i_2}\triangleleft\cdots\triangleleft a_{i_m}\triangleleft b_m.
    \]
    Finally, it takes a common extension of this ordering with the previous condition $o$.

    Since every $P^X_{\textsf{DLO}}$ acts infinitely often, this ensures that $\triangleleft\vert_X$ has no endpoints and that between any two previously ordered elements of $X$ a fresh element of $X$ is eventually inserted. Hence $\textsf{DLO}(\triangleleft\vert_X)$ holds in the final model.
\end{proof}

\section{Variants of the construction}
\label{applications}

At this point, a general recipe for building a model of $\textsf{T}^1_2(\alpha)+\textsf{PHP}(\Delta^b_1(\alpha))$, where $\alpha$ witnesses the failure of some combinatorial principle $C$, can be summarized as follows.
\begin{enumerate}
    \item Define a poset $\mathbb P_C$ containing conditions for building an interpretation of $\alpha$ such that the player $P_C$ aiming to falsify $C$ has a winning strategy.

    \item Verify that $P_{\textsf{IND}}$ has a winning strategy. This amounts to proving a theorem similar to the result of Paris and Wilkie \cite{pariswilkie85}; a general treatment of this kind of problem in a forcing framework can be found in \cite{atseriasmuller15}.

    \item Verify that $P_{\textsf{PHP}}$ has a winning strategy. This is achieved by the following steps.
    \begin{enumerate}
        \item Define a suitable analogue of an $(o,p,h)$-array and show that the existence of such an array is mutually exclusive with the immediate success of the PHP player.
        \item Define a suitable notion of a tree, as in Definition~\ref{min tree}, and extend rows of the array to have tree structure. This gives a lower bound on the size of the array.
        \item Prove an upper bound on the size of a family of pairwise incompatible conditions in $\mathbb P_C$.
        \item Compare the lower and upper bounds to derive an inequality such as $p\leq h$.
    \end{enumerate}
\end{enumerate}

A particular instance of this method, albeit in a slightly different language, was used in \cite{narusevych24} to separate $\textsf{WPHP}(\Delta^b_1(R))$ and $\forall m\leq n^\epsilon\,\textsf{PHP}^{m+1}_m(\Delta^b_1(R))$ from $\textsf{ontoPHP}^{n+1}_n(R)$ over $\textsf{T}^1_2(R)$.

In the following subsections, we sketch additional separations obtained by the same method.

\subsection{The $\TOUR$ principle}

\begin{definition}[{\cite[Section 12.1]{krajicek95}}]
    A \deff{tournament} is a directed graph $(V,E)$ with exactly one directed edge between any two distinct vertices.

    A set $X\subseteq V$ is \deff{dominating} if, for every vertex $w\in V\setminus X$, there is $v\in X$ such that $(v,w)\in E$.
\end{definition}

A well-known fact is that a tournament on $n$ vertices contains a dominating set of size at most $|n|+1$; see, for example, \cite[Section 2.5]{megiddovishkin88}.

Given a binary relation $E$ and a number $n$, the following sentence $\TOUR(E)_n$ expresses this fact:
\begin{gather*}
    \exists x\neq y<n\,(\neg E(x,y)\land\neg E(y,x))
    \\
    \lor
    \\
    \exists x\neq y<n\,(E(x,y)\land E(y,x))
    \\
    \lor
    \\
    \exists X\subseteq[0,\dots,n)\,\bigl(|X|\leq |n|+1\land
    \forall x<n\,(x\notin X\to \exists y\in X\,E(y,x))\bigr).
\end{gather*}
Sets $X$ quantified in the above expression are taken to be definable in the ground model $\M$, and hence codable by numbers of size approximately $n^{|n|+1}$, which belong to $\I$.

\begin{remark}
    The tournament principle was shown by Je\v r\'abek \cite{jerabek08} to be provable in the bounded arithmetic theory $\textsf{APC}_2(E)$.
\end{remark}

Following the recipe above, define $\mathbb T$ as the set of tournaments of size at most $|n|^c$ on subsets of $[0,\dots,n)$, where $c$ is an arbitrary standard constant. The order is inverse inclusion.

There is a simple winning strategy for the player $P_{\textsf{DOM}\vert X}$, whose goal is to ensure that a given small set $X$ is not dominating in the resulting tournament. Given $t\in\mathbb T$, choose $x\in[0,\dots,n)\setminus X$ not yet appearing in $t$, and extend $t$ by adding $x$ with all edges between $x$ and elements of $X$ directed from $x$ towards $X$. Then no vertex from $X$ dominates $x$.

The exact size bound $|X|\leq |n|+1$ is not essential here. It suffices that $X$ has size at most $|n|^c$ for a standard $c$, similarly to the situation in Theorem~\ref{dense theorem}.

To ensure that $E$ is a tournament on the full interval $[0,\dots,n)$, we add another player $P_{\textsf{TOUR}}$, who maintains a countable queue of vertices from $[0,\dots,n)$ not yet inserted into the tournament. Given $t$, the player chooses the first vertex in the queue and adds it to $t$, orienting its edges to the existing vertices arbitrarily.

We omit the treatment of $P_{\textsf{IND}}$, since the proof is analogous to \cite[Theorem 21]{pariswilkie85}.

The definition of a $t$-$\TOUR$-tree is the expected one. Choosing a new vertex $x$, one branches according to all possible ways of extending the current tournament $t$ by including $x$. A uniform $t$-tree of depth $d$ has size
\[
    2^{|t|}\cdot 2^{|t|+1}\cdots 2^{|t|+d-1}.
\]
As in Theorem~\ref{upper bound}, one proves that a pairwise incompatible set of tournament conditions extending $t$ and having size $|t|+d$ has cardinality at most the same product. Hence the lower-bound and upper-bound comparison again yields the dimensional restriction $p\leq h$, which gives the desired PHP strategy.

\subsection{The $\dWPHP$ principle}

The following principle was first considered by A. Wilkie; see \cite[Section 11.2]{krajicek19}. There the principle is referred to as $\textsf{WPHP}$; the current terminology was developed later.

\begin{definition}
    For a relation $R\subseteq[0,\dots,n)\times[0,\dots,2n)$, the \deff{dual weak pigeonhole principle} says that $R$ cannot be the graph of a surjective function from $[0,\dots,n)$ onto $[0,\dots,2n)$.

    This is expressed by the sentence $\dWPHP(R)$:
    \begin{gather*}
        \exists x<n\,\exists y\neq z<2n\,(R(x,y)\land R(x,z))
        \\
        \lor
        \\
        \exists x<n\,\forall y<2n\,\neg R(x,y)
        \\
        \lor
        \\
        \exists y<2n\,\forall x<n\,\neg R(x,y).
    \end{gather*}
\end{definition}

To build a model of $\textsf{T}^1_2(R)$ that violates $\dWPHP(R)$, we play the game with $\mathbb D$ being the set of graphs of partial functions on $[0,\dots,n)\times[0,\dots,2n)$ of size at most $|n|^c$, for some standard constant $c$. The order is inverse inclusion.

The player $P_{\dWPHP}$ maintains a countable queue of elements of $[0,\dots,2n)$ not yet in the range of the partial function. At each move, it extends the current condition so as to put the next element of this queue into the range, while preserving functionality. Thus the final relation is the graph of a total surjection from $[0,\dots,n)$ onto $[0,\dots,2n)$.

A winning strategy for $P_{\textsf{IND}}$ again follows from \cite[Theorem 21]{pariswilkie85}.

Finally, a winning strategy for $P_{\textsf{PHP}}$ follows from a combinatorial analysis similar to that in \cite{narusevych24}. There, the bijective pigeonhole principle was considered. In that case, trees allowed two kinds of queries: one could either choose a pigeon and ask for its corresponding hole, or choose a hole and ask for its corresponding pigeon. In the current context, only the first kind of query is allowed: given a pigeon, ask for its corresponding hole.

The size of a uniform tree of depth $d$ is exactly $(2n)^d$. On the other hand, the upper bound on the size of a depth-$d$ uniform set of pairwise incompatible elements of $\mathbb D$ is also $(2n)^d$. Hence the same array-size comparison gives the desired strategy for $P_{\textsf{PHP}}$.

\subsection{Open problem}

Consider the following formal expression, where $\varphi(x,y)$ is a $\Dmin$ formula and $h<p$:
\begin{gather*}
    \exists x\neq y<p\,\exists z<h\,(\varphi(x,z)\land\varphi(y,z))
    \\
    \lor
    \\
    \exists x<p\,\forall y<h\,\neg\varphi(x,y).
\end{gather*}
This is a variant of the pigeonhole principle that allows multifunctions and is possibly stronger than the $\textsf{PHP}(\varphi(x,y))$ defined in Section~\ref{subsection min}.

Let $\textsf{PHP}'(\Dmin)$ denote the set of instances of the above formal expression for all $\varphi(x,y)\in\Dmin$ and all $h<p$ from $\I$.

We do not know whether it is possible to build a model of
\[
    \Tmin+\textsf{PHP}'(\Dmin)+\neg\MIN_n
\]
using the combinatorial methods of the present paper. The apparent difficulty is that one needs lower bounds on the size of a modified version of an $(o,p,h)$-array in which elements of the same row are no longer necessarily incompatible, or even distinct.

\section*{Acknowledgements}

I thank my supervisor, Jan Kraj\'i\v cek, for numerous suggestions. I am grateful to Ond\v rej Je\v zil, who helped clarify several technical arguments, especially in the final section. Finally, I thank Moritz M\"uller, without whose recommendations the final section might not have appeared.

\bibliographystyle{plain}
\bibliography{main}

\end{document}